\documentclass[11pt,reqno]{amsart}

\usepackage{mathrsfs}

\usepackage{hyperref}

\usepackage{graphicx}

\usepackage{xcolor}

\usepackage{enumerate}
\usepackage{bm}

\newtheorem{theorem}{Theorem}[section]
\newtheorem{lemma}[theorem]{Lemma}

\theoremstyle{definition}
\newtheorem{definition}[theorem]{Definition}

\theoremstyle{remark}
\newtheorem{remark}[theorem]{Remark}

\numberwithin{equation}{section}

\begin{document}

\title{ 
Viscosity solutions to uniformly elliptic complex equations
}

\author{Wei Sun}

\address{Institute of Mathematical Sciences, ShanghaiTech University, Shanghai, China}
\email{sunwei@shanghaitech.edu.cn}




\begin{abstract}

In this paper, we shall extend the definition of $\mathcal{C}$-subsolution condition and adapt the argument of Guo-Phong-Tong~\cite{GPT2021} to replace Alexandroff-Bakelman-Pucci estimate in complex cases. 
As an application, we shall define and study the viscosity solutions to uniformly elliptic complex  equations and prove the H\"older regularity, following the argument for real equations. 
Our results show that the new method can improve the dependence in regularity and a priori estimates for complex elliptic equations.

\end{abstract}

\maketitle

\medskip
\section{Introduction}

In this paper, we shall study the viscosity solutions to complex elliptic equation
\begin{equation}
\label{equation-1-1}
	F(\sqrt{-1}\partial\bar\partial u (\bm{z}) , \bm{z}) = f(\bm{z})
\end{equation}
where $\bm{z} \in \Omega$, $u$ and $f$ are defined in a bounded domain $\Omega \subset \mathbb{C}^n$.
The real-valued function $F (M, \bm{z})$ is defined on $\mathcal{H} \times \Omega$, where $\mathcal{H}$ is the space of $n\times n$ Hermitian matrices. 
Moreover, $F$ is assumed to be uniformly elliptic for Hermitian matrices as defined below. 
\begin{definition}[Uniform ellipticity]

$F$ is uniformly elliptic if there are constants $\Lambda \geq \lambda > 0$ such that for any $M \in \mathcal{H}$ and $\bm{z} \in \Omega$
\begin{equation*}
	\lambda \Vert N \Vert \leq F (M + N, \bm{z}) - F(M,\bm{z}) \leq \Lambda \Vert N \Vert , \qquad \forall N \geq 0 , 
\end{equation*}
where
\begin{equation*}
	\Vert N \Vert := \sup_{|\bm{v}|  = 1} |N \bm{v}|   = \sup_{\bm{v} \neq \bm{0}} \frac{\sqrt{\langle N \bm{v} , N \bm{v}\rangle}}{\sqrt{\langle\bm{v},\bm{v}\rangle}} 
\end{equation*}
which is the maximum value of eigenvalues of $N$.

\end{definition}

For convenience, the definition for viscosity solutions to complex equations can be slightly rephrased as follows. 
\begin{definition}[Viscosity solution] 
   
We define the solution to Equation~\eqref{equation-1-1} in the viscosity sense as follows:
\begin{enumerate}
\item $u \in C(\Omega)$ is a {\em viscosity subsolution} to Equation~\eqref{equation-1-1} if
\begin{equation*}
	F(\sqrt{-1} \partial\bar\partial \varphi (\bm{z_0}), \bm{z_0}) \geq f(\bm{z_0}) 
\end{equation*}
for any $\bm{z}_0 \in \Omega$ and $ \varphi \in C^2 (\Omega)$ such that $u - \varphi$ has a local maximum at $\bm{z_0}$; 

\item $u \in C(\Omega)$ is a {\em viscosity supsolution} to Equation~\eqref{equation-1-1} if
\begin{equation*}
	F(\sqrt{-1} \partial\bar\partial \varphi (\bm{z_0}), \bm{z_0}) \leq f(\bm{z_0}) 
\end{equation*}
for any $ \bm{z_0} \in \Omega$ and $ \varphi \in C^2 (\Omega)$ such that $u - \varphi$ has a local minimum at $\bm{z_0}$;

\item $u \in C(\Omega)$ is a {\em viscosity solution} to Equation~\eqref{equation-1-1} if
$u$ is both a viscosity subsolution and a viscosity supersolution. 
\end{enumerate}

\end{definition}
From the definitions, some interesting results can be derived, which are almost the same as those in real cases. 
For more details, the author would like to refer the readers to the books by Caffarelli-Cabr\'e~\cite{CaffarelliCabre} and Han-Lin~\cite{HanLin1}, which most of our arguments here follow.

In this paper, we can obtain H\"older regularity for viscosity solutions.
\begin{theorem}
\label{theorem-1-3}
Suppose that $p > n$ and $u$ is a viscosity solution to Equation~\eqref{equation-1-1} in  $B_1 (\bm{0})$ for some function $f \in C (B_1 (\bm{0}))$. Then there is a constant $\alpha \in (0,1)$ depending on $n$, $p$, $\lambda$ and $\Lambda$ such that
\begin{equation}
\label{inequality-1-6}
	|u (\bm{z}) - u (\bm{w})| \leq C (n,p,\lambda,\Lambda) |\bm{z} - \bm{w}|^\alpha \left( \Vert u \Vert_{L^2 (B_1 (\bm{0}))} + \Vert f \Vert_{L^p (B_1 (\bm{0}))}\right)
\end{equation}
for $\bm{z}, \bm{w} \in B_{\frac{1}{2}} (\bm{0})$.

\end{theorem}

The H\"older regularity for viscosity solutions  can be achieved via Alexandroff-Bakelman-Pucci estimate. However, the regularity will be formulated by
\begin{equation}
	|u (\bm{z}) - u (\bm{w})| \leq C |\bm{z} - \bm{w}|^\alpha \left( osc_{B_1 (\bm{0})} u + \Vert f \Vert_{L^{2n} (B_1 (\bm{0}))}\right) .
\end{equation}
We observe that the H\"older regularity formula~\ref{inequality-1-6} is close to that for generalized solutions to linear elliptic equations in divergence form. 
Indeed, we can also have a modulus of continuity when $f$ is less regular. 
In order to obtain the new H\"older regularity~\eqref{inequality-1-6}, we need an analogue to Alexandroff-Bakelman-Pucci estimate.

\begin{theorem}
\label{theorem-1-4}
Let $f$ be a continuous and bounded function in $ B_1(\bm{0})$.  
Suppose that $u \in C (\bar B_1 (\bm{0}))$ satisfies that $u \in \overline{\mathcal{S}} (\lambda,\Lambda,f)$ in $B_1(\bm{0})$ and $u \geq 0$ on $\partial B_1(\bm{0})$. Then 
\begin{equation} 
	\sup_{B_1(\bm{0})} (- u )^+ \leq \frac{C}{\lambda} 	\left( \int_{\left\{ u (\bm{z}) < 0\right\}}    (f^+)^p  \beta^n	\right)^{\frac{1}{p}} , 
\end{equation}
where $p > n$ and $\beta$ is a Hermitian structure on $\mathbb{C}^n$, i.e.,
\begin{equation*}
	\beta = \sqrt{-1} \sum_i  d z^i \wedge d \bar z^i .
\end{equation*}

\end{theorem}
For simplicity and convenience in calculation, we write $\sqrt{-1}$ instead of $\frac{\sqrt{-1}}{2}$, which has no influence on our results. 
Indeed, Theorem~\ref{theorem-1-4} can be extended to arbitrary bounded domain $\Omega$ by a comparison principle argument as in Wang-Wang-Zhou~\cite{WangWangZhou2020}\cite{WangWangZhou2021}.

%

Theorem~\ref{theorem-1-4} is a consequence of discovering $\mathcal{C}$-subsolution and then adapting the argument of Guo-Phong-Tong\cite{GPT2021}. 
The $\mathcal{C}$-subsolution condition can be stated as follows.
\begin{definition}[$\mathcal{C}$-subsolution condition]
\label{definition-1-5}

Let $(M, \omega)$ be a Hermitian manifold  of complex dimension $n \geq 2$. For a complex elliptic equation
\begin{equation}
\label{equation-1-10}
	F (\chi + \sqrt{-1} \partial\bar\partial u) = f,
\end{equation}
where $\chi$ is a real $(1,1)$-form.

We say that the real $(1,1)$-form $\tilde \chi$ satisfies the $\mathcal{C}$-subsolution condition with respect to $\mathcal{F}$ if 
$\mathcal{F}$ is a real-valued function on $M$ such that at any point $\bm{z} \in M$ with $(1,1)$-form $\hat\chi (\bm{z}) \geq 0$ satisfying
\begin{equation*}
F \left(\tilde \chi (\bm{z}) + \hat\chi(\bm{z})\right) \leq f (\bm{z}),
\end{equation*}
we must have
\begin{equation}
\label{inequality-1-12}
	\hat\chi^n (\bm{z}) \leq e^{\mathcal{F}(\bm{z})} \omega^n (\bm{z}) .
\end{equation}

\end{definition}
\begin{remark}
In the study of Equation~\ref{equation-1-1}, we know that $\chi = \tilde \chi \equiv 0$ and $\omega = \beta$.
\end{remark}

The original $\mathcal{C}$-subsolution was introduced by Sz\'ekelyhidi~\cite{Szekelyhidi2018} on the basis of the cone condition~\cite{SongWeinkove2008}\cite{FangLaiMa2011}\cite{Sun201701} for complex Monge-Amp\`ere type equations on compact K\"ahler manifolds without boundary, that is, 
\begin{equation*}
	n \chi^{n - 1} > (n - k) c \chi^{n - k - 1} \wedge \omega^k \qquad \text{with } c = \frac{\int_M \chi^n}{\int_M \chi^{n - k} \wedge \omega^k}.
\end{equation*} 
The original $\mathcal{C}$-subsolution condition states that
there are constants $\delta > 0$ and $C > 0$ such that at any point $\bm{z} \in M$ with $(1,1)$-form $\hat\chi (\bm{z}) \geq 0$ satisfying
\begin{equation*}
F \left(\chi (\bm{z}) - \delta \omega(\bm{z}) + \hat\chi(\bm{z})\right) \leq f (\bm{z}),
\end{equation*}
we must have
\begin{equation*}
	\hat\chi (\bm{z}) \leq C \omega (\bm{z}) .
\end{equation*}
It is easy to see that
$\chi - \delta\omega$ satisfies the $\mathcal{C}$-subsolution with respect to some constant function in this case. 
A similar notion was introduced by Guan~\cite{Guan2014} earlier for real problems, and Sz\'ekelyhidi~\cite{Szekelyhidi2018} discussed the relationship between the two notions. 
Assuming the original $\mathcal{C}$-subsolution, Sz\'ekelyhidi adopted the inequality as in Blocki~\cite{Blocki2005}
\begin{equation*}
	\det (D^2 v) \leq 2^{2n} \det (v_{i\bar j})^2
\end{equation*}
and hence applied a variant of the Alexandroff-Bakelman-Pucci maximum principle to obtain the $L^\infty$ estimate. 
However, this method cannot deal with some degenerate or singular cases, e.g. the boundary case
\begin{equation*}
	n \chi^{n - 1} \geq (n - k) c \chi^{n - k - 1} \wedge \omega^k \qquad \text{with } c = \frac{\int_M \chi^n}{\int_M \chi^{n - k} \wedge \omega^k}.
\end{equation*}
For Donaldson equation on K\"ahler surface
\begin{equation*}
(\chi + \sqrt{-1} \partial\bar\partial u)^2 = c (\chi + \sqrt{-1} \partial\bar\partial u) \wedge \omega,
\end{equation*}
Song and Weinkove~\cite{SongWeinkove2008} showed that the quantity $|u| + |\Delta_\omega u|$ may blow up, and later Fang, Lai, Song and Weinkove~\cite{FangLaiSongWeinkove2014} studied the $J$-flow in the boundary case. 
The trick is to rewrite Donaldson equation as a complex Monge-Amp\`ere equation as observed by Chen~\cite{Chen2000}:
\begin{equation}
\label{equation-1-19}
\left(\chi - \frac{c}{2} \omega + \sqrt{-1} \partial\bar\partial u\right)^2 = \frac{c^2}{4} \omega^2 .
\end{equation}
B. Weinkove introduced the boundary case to the author after the author solved the $L^\infty$ estimate via Moser iteration~\cite{Sun201701}. Unfortunately, neither the transformation~\eqref{equation-1-19} nor the method in \cite{Sun201701} works well for the boundary case in higher dimensions. 
In the research of the boundary case, we realized that it is possible and also better to impose some geometric assumption~\eqref{inequality-1-12}, in the view point of PDE. 
For simplicity, we continue to use the name $\mathcal{C}$-subsolution for Definition~\ref{definition-1-5}, where $\mathcal{C}$ means {\em cone} probably. 

Since condition~\eqref{inequality-1-12} is a complex Monge-Amp\`ere operator, we need some more assumptions in addition to the $\mathcal{C}$-subsolution condition, 
on compact K\"ahler manifolds without boundary. 
We assume, at least, that $[\chi - \tilde \chi]$  is nef and big, that is,
\begin{equation}
\label{inequality-1-20}
	[\chi - \tilde \chi + t\omega] \text{ is K\"ahler for any } t > 0, 
	\qquad \text{and }\qquad 
	\int_M (\chi - \tilde \chi)^n > 0
	.
\end{equation} 
Assumption~\eqref{inequality-1-20} covers the following particular cases: 
\begin{enumerate}[(i)]
\item $[\chi - \tilde \chi]$ is semipositive and big; 
\item $[\chi - \tilde \chi]$ is K\"ahler. 
\end{enumerate}
With these assumptions, Sui and the author~\cite{SuiSun2023} discovered the $L^\infty$ estimate for complex Hessian quotient equations by adapting a technique of Guo-Phong-Tong~\cite{GPT2021}, which is to be discussed later. 
Indeed, we found that it is a natural and better replacement for Alexandroff-Bakelman-Pucci estimate to combine $\mathcal{C}$-subsolution condition and Guo-Phong-Tong's argument in complex problems, which is to be discussed later. 
In the study of complex elliptic equations~\cite{Sun202210}\cite{Sun202211}, it is a key to find out some $\mathcal{C}$-subsolution conditions.
In this paper, we plan to study the H\"older regularity of the viscosity solutions to uniformly elliptic complex equations on $\mathbb{C}^n$ in this way, which is able to show a better dependence than Alexandroff-Bakelman-Pucci estimate.
However, it is usually an open problem to find $\tilde \chi$ satisfying \eqref{inequality-1-20} on complex manifolds. In views of \cite{DemaillyPaun2004}\cite{Chen2021}, we plan to discover some numerical conditions to find appropriate $\tilde \chi$ for complex Hessian quotient equations in future.

To utilize assumption~\eqref{inequality-1-12}, we shall investigate  a complex Monge-Amp\`ere equation.  
In particular,  the techniques in $L^\infty$ estimate of complex Monge-Amp\`ere equations are heuristic and useful for other complex elliptic equations.
It is well known that   pluripotential theory is a powerful technique to study  complex Monge-Amp\`ere equation, 
since the pioneering works of Bedford and Taylor~\cite{BedfordTaylor1976}\cite{BedfordTaylor1982}. 
%
%
%
%
%
%
Besides pluripotential theory, there are also some PDE proofs for $L^\infty$ estimate. Yau~\cite{Yau1978} applied Moser iteration when the equation is smooth, and Blocki~\cite{Blocki2005} adopted Aleksandrov-Bakelman-Pucci estimate with the right-hand side of the equation in $L^q$ $(q > 2)$.
For Dirichlet problems on $\mathbb{C}^n$, Wang-Wang-Zhou~\cite{WangWangZhou2021} used Sobolev type inequality~\cite{WangWangZhou2020} to give a proof. 
Later, Guo-Phong-Tong~\cite{GPT2021} combined the methods of Wang-Wang-Zhou~\cite{WangWangZhou2021} and Chen-Cheng~\cite{ChenCheng2021}.
In fact, B. Guan reminded the author of \cite{WangWangZhou2021}. But it seems to the author that their method and related functionals are not compatible with the $\mathcal{C}$-subsolution condition~\eqref{inequality-1-12} on complex manifolds. 
In this paper, the author would like to adapt the technique of Guo-Phong-Tong~\cite{GPT2021} to $\mathcal{C}$-subsolution condition. 
But, the author believes that pluripotential theory can also work well with $\mathcal{C}$-subsolution condition.

In \cite{GP2022}\cite{GP202207}\cite{GPT2021}\cite{GPT202106}\cite{GPTW20212}, the complex elliptic operator in consideration is defined by
\begin{equation}
\label{equation-1-21}
 F (\chi + \sqrt{-1}\partial\bar\partial u) 
 := 
 \mathfrak{f} (\bm{\lambda} (\chi + \sqrt{-1} \partial\bar\partial u)) 
 ,
\end{equation}
where $[\chi]$ is nef and big, and $\bm{\lambda} (\chi + \sqrt{-1} \partial\bar\partial u)$ is the eigenvalue vector of $\chi + \sqrt{-1} \partial\bar\partial u$ with respect to $\Omega$.
Real valued function $\mathfrak{f}$  is supposed to satisfy the following conditions:
\begin{enumerate}
\item $\mathfrak{f} (\bm{\lambda})$ is defined on $\Gamma \subset \mathbb{R}^n$, where $\Gamma$ is an open symmetric convex cone with vertex at the origin containing the positive cone $\Gamma^n := \{\bm{\lambda} \in \mathbb{R}^n | \lambda_1 > 0, \cdots, \lambda_n > 0\}$;

\item $\mathfrak{f} (\bm{\lambda})$ is invariant under permutations of the components of $\bm{\lambda}$;

\item for any $1 \leq i \leq n$ and $\bm{\lambda} \in \Gamma$,
\begin{equation*}
\frac{\partial \mathfrak{f}}{\partial \lambda_i} (\bm{\lambda}) > 0 ;
\end{equation*}

\item there is a constant $c > 0$ such that for any $\bm{\lambda} \in \Gamma$,
\begin{equation*}
	\prod^n_{i = 0} \frac{\partial \mathfrak{f}}{\partial \lambda_i} (\bm{\lambda}) \geq c ;
\end{equation*}

\item there is a constant $C > 0$ such that for any $\bm{\lambda} \in \Gamma$,
\begin{equation*}
	\sum^n_{i = 1} \frac{\partial \mathfrak{f}}{\partial \lambda_i} (\bm{\lambda}) \lambda_i \leq C \mathfrak{f} (\bm{\lambda}) .
\end{equation*}

\end{enumerate}
Complex Hessian equations, including complex Monge-Amp\`ere equation, satisfy these conditions, while complex Hessian quotient equations do not. 
As described in \cite{SuiSun2023}, these conditions imply that $\chi$ satisfies the $\mathcal{C}$-subsolution condition with respect to $\frac{1}{c} \left(\frac{C f}{n}\right)^n$. Our idea is to find out an appropriate $\mathcal{C}$-subsolution condition from elliptic operator $F$ and then adapt the argument of Guo-Phong-Tong to the Monge-Amp\`ere operator~\eqref{inequality-1-12} directly, as in~\cite{SuiSun2023}\cite{Sun202210}\cite{Sun202211}. 

\medskip
\section{Preliminary}

In this section, we shall state some notations and lemmas, possibly with some little modifications subject to complex problems. For details, we refer the readers to textbooks on elliptic partial differential equations of second order, 
e.g. Gilbarg-Trudinger~\cite{GilbargTrudinger}, Caffarelli-Cabr\'e~\cite{CaffarelliCabre}, Chen-Wu~\cite{ChenWu}  and  Han-Lin~\cite{HanLin1}.

\medskip
\subsection{Pucci's extremal operators}

Following Caffarelli-Cabr\'e~\cite{CaffarelliCabre}, we shall adapt the definitions on Pucci's extremal operators for the complex cases.

\begin{definition}[Pucci's extremal operators in complex case]

For $M \in \mathcal{H}$,
\begin{equation*}
\mathcal{M^-} (M, \lambda,\Lambda) 
:= \lambda \sum_{e_i > 0} e_i + \Lambda \sum_{e_i < 0} e_i ,
\end{equation*}
and
\begin{equation*}
\mathcal{M^+} (M,\lambda,\Lambda)
:= \Lambda \sum_{e_i > 0} e_i + \lambda \sum_{e_i > 0} e_i ,
\end{equation*}
where 
$\{e_i | i = 1,\cdots ,n\}$ is the eigenvalue set of $M \in \mathcal{H}$ and $0 < \lambda \leq \Lambda$.

\end{definition}

\begin{definition}
Suppose that real valued function $f \in C(\Omega)$ and $0 < \lambda \leq \Lambda$. We denote
\begin{equation*}
\underline{\mathcal{S}} (\lambda,\Lambda,f) := \left\{ u \in C (\Omega) | \mathcal{M}^+ (\sqrt{-1}\partial\bar\partial u,\lambda,\Lambda) \geq f(x) \text{ in the viscosity sense}\right\} ,
\end{equation*}
\begin{equation*}
\overline{\mathcal{S}} (\lambda,\Lambda,f) := \left\{ u \in C (\Omega) | \mathcal{M}^- (\sqrt{-1}\partial\bar\partial u,\lambda,\Lambda) \leq f(x) \text{ in the viscosity sense}\right\} ,
\end{equation*}
and hence
\begin{equation*}
{\mathcal{S}} (\lambda,\Lambda,f) := \underline{\mathcal{S}} (\lambda,\Lambda,f) \cap \overline{\mathcal{S}} (\lambda,\Lambda,f) .
\end{equation*}

\end{definition}

\medskip

\subsection{Some important lemmas}

We shall need some iteration methods in the arguments. De Giorgi iteration will be applied in $L^\infty$ estimate in Section~\ref{Analogue}.  Here we state a version from Chen-Wu~\cite{ChenWu}. 
\begin{lemma}[De Giorgi iteration]
\label{lemma-3-1}

Suppose that $\phi (s)$ is a nonnegaive increasing function on $[s_0,+\infty]$ such that
\begin{equation*}
	s' \phi (s' + s) \leq C_0 \phi^{1+\delta} (s), \qquad \forall s'>0, s \geq s_0,
\end{equation*}
where $\delta > 0$. Then $\phi (s_0 + d) = 0$ when $d \geq 2^{\frac{1 + \delta}{\delta}} C_0 \phi^{\delta} (s_0) $.

\end{lemma}

A standard way to derive H\"older regularity is to apply the following lemma to a certain Harnack inequality. To well calculate the exponents and coefficients later, we write down the lemma from Gilbarg-Trudinger~\cite{GilbargTrudinger}. 
\begin{lemma}
\label{lemma-2-7}
Suppose that $\omega$ and $\sigma$ are nondecreasing functions on an interval $(0,R]$. 
If there are positive constants $ \gamma, \tau < 1$ satisfying
\begin{equation}
\label{inequality-2-1}
\omega (\tau r) \leq \gamma \omega (r) + \sigma (r), \qquad \forall\, 0 < r \leq R, 
\end{equation}
then we have that for any $\mu \in (0,1)$ and $r \leq R$ 
\begin{equation*}
	\omega (r) \leq C (\gamma,\tau) \left\{ \left( \frac{r}{R} \right)^\alpha  \omega (R) + \sigma \left(r^{\mu} R^{1-\mu}\right)     \right\}
\end{equation*}
where $\alpha := (1 - \mu) \frac{\ln \gamma}{\ln \tau}$.

\end{lemma}

In Section~\ref{Harnack}, the main task is to obtain a Harnack inequality in the form of \eqref{inequality-2-1}. Moreover, we shall also need Calderon-Zygmund decomposition. Nevertheless, the procedure and the coefficients are the same as those for real equations, so we omit the statements.


\medskip
\section{Analogue to Alexandroff-Bakelman-Pucci estimate}
\label{Analogue}

In this section, we shall prove Theorem~\ref{theorem-1-4}, by adapting the argument of Guo-Phong-Tong~\cite{GPT2021}. 
In the arguments, $C$ denotes a constant, which might vary.

%
%
%
%
%
%
%
%
%
%
%
%
%

Define
\begin{equation*}
	\Omega_s := \{\bm{z} \in B_1(\bm{0})  | - u (\bm{z}) > s\} \subset B_1 (\bm{0}) .
\end{equation*}
We solve the Dirichlet problem
\begin{equation}
\label{equation-3-3}
\left\{
\begin{aligned}
	\left(\sqrt{-1} \partial\bar\partial \psi_{s,k}\right)^n &= \frac{\tau_k (- u - s)}{A_{s,k}} \mathcal{F}_k \beta^n , && \text{in } B_1(\bm{0}) \\
	\psi_{s,k} &= 0 , && \text{on } \partial B_1(\bm{0}) ,
\end{aligned}
\right.
\end{equation}
where
\begin{equation*}
	A_{s,k} := \int_{B_1 (\bm{0})}   \tau_k (- u - s)  \mathcal{F}_k \beta^n \to A_s := \int_{\Omega_s} (- u - s) \left((f^+)^n +\epsilon\right) \beta^n .
\end{equation*}
$\tau_k : \mathbb{R} \to \mathbb{R}^+$ is a decreasing sequence of smooth functions such that
\begin{equation*}
\tau_k (t) = 
\left\{
\begin{aligned}
&t + \frac{3}{k} , && \text{when } x \geq - \frac{1}{k}, \\
& \frac{1}{k} , && \text{when } x \leq - \frac{2}{k}
\end{aligned}
\right.
\end{equation*}
and otherwise
\begin{equation*}
	\frac{1}{k} \leq \tau_k (t) \leq \frac{2}{k} .
\end{equation*}
$\mathcal{F}_k$ is a uniformly decreasing sequence of smooth functions such that
\begin{equation*}
	0 < \mathcal{F}_k (\bm{z}) - (f^+)^n (\bm{z}) - \epsilon < \frac{1}{k} ,
\end{equation*}
where $\epsilon > 0$. 
By \cite{CaffarelliKohnNirenbergSpruck1985}, the solution $\psi_{s,k}$ is smooth.

Define a continuous function $\Phi$ as follows:
\begin{equation*}
	\Phi := - \left(\frac{n^2  \lambda}{n + 1}\right)^{- \frac{n}{n + 1}}   A_{s,k}^{\frac{1}{n + 1}} \left(- \psi_{s,k}  \right)^{\frac{n}{n + 1}} - u - s ,
\end{equation*}
where $s > 0$.
If $\Phi$ reaches its maximal value at $\bm{z_{max}} \in \bar B_1 (\bm{0}) \setminus \Omega_s$, 
\begin{equation*}
	\Phi (\bm{z_{max}}) \leq  - u - s \leq 0.
\end{equation*}
If $\Phi$ reaches its maximal value at $\bm{z_{max}} \in \Omega_s$, then 
$
u + \left(\frac{n^2 \lambda}{n + 1}\right)^{- \frac{n}{n + 1}} A^{\frac{1}{n + 1}}_{s,k} (- \psi_{s,k})^{\frac{n}{n + 1}}  + s
$
has a local minimum at $\bm{z_{max}}$. 
According to the definitions,
\begin{equation*}
\begin{aligned}
\mathcal{M}^- \left(- \left(\frac{n^2 \lambda}{n + 1}\right)^{-\frac{n}{n + 1}} A^{\frac{1}{n + 1}}_{s,k}  \sqrt{-1} \partial\bar\partial (- \psi_{s,k})^{\frac{n}{n + 1}},\lambda,\Lambda\right)	\leq f = f^+
\end{aligned}
\end{equation*}
at point $\bm{z_{max}}$. 
Direct calculation shows that
\begin{equation*}
\begin{aligned}
	&\quad - \left(\frac{n^2 \lambda}{n + 1}\right)^{-\frac{n}{n + 1}} A^{\frac{1}{n + 1}}_{s,k} \sqrt{-1} \partial\bar\partial (- \psi_{s,k})^{\frac{n}{n + 1}} \\
	&\geq \left(\frac{n^2 \lambda}{n + 1}\right)^{-\frac{n}{n + 1}} \frac{n}{n + 1}  A^{\frac{1}{n + 1}}_{s,k}  \left(-\psi_{s,k}\right)^{- \frac{1}{n + 1}} \sqrt{-1}\partial\bar\partial\psi_{s,k}     
	\geq 0 ,
\end{aligned}
\end{equation*}
and hence
\begin{equation}
\label{inequality-3-16}
\begin{aligned}
	f^+ 
	&\geq \mathcal{M}^- \left(- \left(\frac{n^2 \lambda}{n + 1}\right)^{-\frac{n}{n + 1}} A^{\frac{1}{n + 1}}_{s,k}  \sqrt{-1} \partial\bar\partial (- \psi_{s,k})^{\frac{n}{n + 1}},\lambda,\Lambda\right) \\
	&\geq \left(\frac{n^2 \lambda}{n + 1}\right)^{-\frac{n}{n + 1}} \frac{n}{n + 1}  A^{\frac{1}{n + 1}}_{s,k}  \left(-\psi_{s,k}\right)^{- \frac{1}{n + 1}} \lambda Tr \left( \sqrt{-1}\partial\bar\partial\psi_{s,k}  \right) \\
	&\geq 
	\left(\frac{n^2 \lambda}{n + 1}\right)^{\frac{1}{n + 1}}   A^{\frac{1}{n + 1}}_{s,k}  \left(-\psi_{s,k}\right)^{- \frac{1}{n + 1}} \left(\det \left( \sqrt{-1}\partial\bar\partial\psi_{s,k}  \right)\right)^{\frac{1}{n}} \\
	&\geq 
	\left(\frac{n^2 \lambda}{n + 1}\right)^{\frac{1}{n + 1}}   A^{- \frac{1}{n (n + 1)}}_{s,k}  \left(-\psi_{s,k}\right)^{- \frac{1}{n + 1}}   (- u - s)^{\frac{1}{n}}  f^+ 
	.
\end{aligned}
\end{equation}
Taking the $n$-th power of \eqref{inequality-3-16}, we have that at $\bm{z_{max}}$
\begin{equation*}
\begin{aligned}
	1
	&\geq 
	\left(\frac{n^2 \lambda}{n + 1}\right)^{\frac{n}{n + 1}}   A^{- \frac{1}{n + 1}}_{s,k}  \left(-\psi_{s,k}\right)^{- \frac{n}{n + 1}}   (- u - s)
	,
\end{aligned}
\end{equation*}
that is,
\begin{equation*}
\Phi = - \left(\frac{n^2  \lambda}{n + 1}\right)^{- \frac{n}{n + 1}}   A_{s,k}^{\frac{1}{n + 1}} \left(- \psi_{s,k}  \right)^{\frac{n}{n + 1}} - u - s 
\leq 0
.
\end{equation*}
In sum, $\Phi \leq 0$ in $\bar B_1 (\bm{0})$.

\begin{theorem}
\label{theorem-3-2}
Let $f$ be a continuous and bounded function in $ B_1(\bm{0})$.  
Suppose that $u \in C (B_1 (\bm{0}))$ satisfies that $u \in \overline{\mathcal{S}} (\lambda,\Lambda,f)$ in $B_1(\bm{0})$ and $u \geq 0$ on $\partial B_1(\bm{0})$. Then 
\begin{equation*} 
	\sup_{B_1(\bm{0})} (- u )^+ \leq \frac{C}{\lambda} 	\left( \int_{\left\{ u (\bm{z}) < 0\right\}}  \left( (f^+)^n \right)   \ln^p \left(1 +  (f^+)^n   \right)  \beta^n	+ 1	\right)^{\frac{1}{p}}  \left(\int_{\left\{ u (\bm{z}) < 0\right\}}  (f^+)^n   \beta^n\right)^{\frac{1}{n} - \frac{1}{p}}, 
\end{equation*}
where $p > n$.
\end{theorem}

\begin{proof}

The result of \cite{Kolodziej1998}\cite{WangWangZhou2020} implies that, 
there are positive constants $\alpha$ and $C$ such that
\begin{equation*}
\int_{\Omega_s} \exp \left( \frac{\alpha n^2 \lambda}{n + 1}  A^{- \frac{1}{n}}_{s,k} (- u - s)^{\frac{n + 1}{n}}\right) \beta^n
\leq 
\int_{B_1 (\bm{0})} \exp \left( - \alpha \psi_{s,k}\right) \beta^n \leq C ,
\end{equation*}
where $C$ is independent of $\epsilon$. 
Letting $k \to \infty$, we have
\begin{equation}
\label{inequality-3-20}
\int_{\Omega_s} \exp \left( \frac{\alpha n^2 \lambda}{n + 1}  A^{- \frac{1}{n}}_{s} (- u - s)^{\frac{n + 1}{n}}\right) \beta^n \leq C .
\end{equation}

By H\"older inequality with respect to measure $(f^+)^n \beta^n$,
\begin{equation}
\label{inequality-3-21}
\begin{aligned}
	A_s
	&= \int_{\Omega_s} (- u - s) \left((f^+)^n +\epsilon \right)\beta^n \\
	&\leq
	\left(\int_{\Omega_s}    (-u - s)^{\frac{(n + 1) p}{n}} \left( (f^+)^n +\epsilon \right) \beta^n\right)^{\frac{n}{(n + 1) p}} \left(\int_{\Omega_s}  \left( (f^+)^n +\epsilon\right) \beta^n  \right)^{\frac{(n + 1)p - n}{(n + 1) p}} .
\end{aligned}
\end{equation}
Applying generalized Young's inequality and \eqref{inequality-3-20} to \eqref{inequality-3-21},
\begin{equation}
\label{inequality-3-22}
\begin{aligned}
	A_s
	&\leq 
	A^{\frac{1}{n + 1}}_s \left(\frac{\alpha n^2 \lambda}{2(n + 1)}  \right)^{-\frac{n}{n + 1}}
		\left( \int_{\Omega_0}  \left( (f^+)^n +\epsilon\right)  \ln^p \left(1 +  (f^+)^n +\epsilon \right)  \beta^n
		+ C	\right)^{\frac{n}{(n + 1) p}} \\
		&\qquad 
		\cdot \left(\int_{\Omega_s}   \left( (f^+)^n +\epsilon\right)  \beta^n  \right)^{\frac{(n + 1)p - n}{(n + 1) p}} 
		,
\end{aligned}
\end{equation}
where $C$ is independent of $\epsilon$.
Rewriting \eqref{inequality-3-22},
\begin{equation*}
\begin{aligned}
	A_s
	&\leq
		\frac{2(n + 1)}{\alpha n^2 \lambda}
			\left( \int_{\Omega_0}  \left( (f^+)^n +\epsilon\right)   \ln^p \left(1 +  (f^+)^n +\epsilon  \right)  \beta^n
			+ C	\right)^{\frac{1}{p}} \\  
			&\qquad \qquad\qquad  \cdot\left(\int_{\Omega_s} 
			 \left( (f^+)^n +\epsilon\right)  \beta^n  \right)^{\frac{(n + 1)p - n}{n p}} 
			.
\end{aligned}
\end{equation*}
For any $s',s > 0$,
\begin{equation*}
\begin{aligned}
	s' \int_{\Omega_{s + s'}} \left((f^+)^n + \epsilon\right) \beta^n 
	&\leq
	\int_{\Omega_s } (- u - s) \left((f^+)^n + \epsilon\right) \beta^n \\	
	&\leq
	\frac{2(n + 1)}{\alpha n^2 \lambda}
		\left( \int_{\Omega_0}  \left( (f^+)^n +\epsilon\right)   \ln^p \left(1 +  (f^+)^n +\epsilon  \right)  \beta^n	+ C	\right)^{\frac{1}{p}} \\
		&\qquad\qquad \qquad \cdot  \left(\int_{\Omega_s}   \left( (f^+)^n +\epsilon\right)  \beta^n  \right)^{\frac{(n + 1)p - n}{n p}} 
\end{aligned}
\end{equation*}
By De Giorgi iteration~(Lemma~\ref{lemma-3-1}),
\begin{equation*}
\int_{\Omega_{d}} \left( (f^+)^n + \epsilon\right) \beta^n = 0 ,
\end{equation*}
when 
\begin{equation*}
\begin{aligned}
	d 
	&\geq 
	2^{\frac{np + 2p - 2n}{p - n}} \frac{n + 1}{\alpha n^2 \lambda}
			\left( \int_{\Omega_0}  \left( (f^+)^n +\epsilon\right)   \ln^p \left(1 +  (f^+)^n +\epsilon  \right)  \beta^n		+ C	\right)^{\frac{1}{p}}  \\
			&\qquad \qquad \qquad \qquad\qquad \cdot\left(\int_{\Omega_0} \left( (f^+)^n + \epsilon\right) \beta^n\right)^{\frac{1}{n} - \frac{1}{p}}
			.
\end{aligned}
\end{equation*}
Since $\epsilon > 0$ is arbitrary and $C$ is independent of $\epsilon$, we obtain that 
\begin{equation*}
	- u \leq 
		2^{\frac{np + 2p - 2n}{p - n}} \frac{n + 1}{\alpha n^2 \lambda}
				\left( \int_{\Omega_0}  \left( (f^+)^n \right)   \ln^p \left(1 +  (f^+)^n   \right)  \beta^n		+ C	\right)^{\frac{1}{p}}  \left(\int_{\Omega_0} (f^+)^n   \beta^n\right)^{\frac{1}{n} - \frac{1}{p}}
				.
\end{equation*}

\end{proof}

Now we prove Theorem~\ref{theorem-1-4} from an intermediate step~\eqref{inequality-3-21} of the argument above.
\begin{proof}[Proof of Theorem~\ref{theorem-1-4}]

Following the proof of Theorem~\ref{theorem-3-2}, 
\begin{equation*}
\begin{aligned}
	A_s
	&=
	\int_{\Omega_s} (- u - s) \left((f^+)^n + \epsilon\right) \beta^n \\
	&= 
	\left(\frac{\alpha n^2 \lambda}{n + 1}\right)^{- \frac{n}{n + 1}} A^{\frac{1}{n + 1}}_s \int_{\Omega_s} \left( \frac{\alpha n^2 \lambda}{n + 1} A^{- \frac{1}{n}}_s (- u - s)^{\frac{n + 1}{n}}\right)^{\frac{n}{n + 1}} \left((f^+)^n + \epsilon\right) \beta^n 
	,
\end{aligned}
\end{equation*}
and hence
\begin{equation}
\label{inequality-3-30}
\begin{aligned}
	A_s 
	= 	
	 \frac{n + 1}{\alpha n^2 \lambda}   \left( \int_{\Omega_s} \left( \frac{\alpha n^2 \lambda}{n + 1} A^{- \frac{1}{n}}_s (- u - s)^{\frac{n + 1}{n}}\right)^{\frac{n}{n + 1}} \left((f^+)^n + \epsilon\right) \beta^n \right)^{\frac{n + 1}{n}} .
\end{aligned}
\end{equation}
Applying H\"older inequality with respect to measure $\beta^n$, we obtain that
\begin{equation}
\label{inequality-3-31}
\begin{aligned}
	&\quad \int_{\Omega_s} \left( \frac{\alpha n^2 \lambda}{n + 1} A^{- \frac{1}{n}}_s (- u - s)^{\frac{n + 1}{n}}\right)^{\frac{n}{n + 1}} \left((f^+)^n + \epsilon\right) \beta^n  \\
	&\leq 
	\left(\int_{\Omega_s} \left(\frac{\alpha n^2 \lambda}{n + 1} A^{- \frac{1}{n}}_s (- u - s)^{\frac{n + 1}{n}}\right)^{\frac{n p^2}{(n + 1)p^2 - n p + n^2}} \beta^n\right)^{1 - \frac{n (p - n)}{(n + 1) p^2}} \\
	&\qquad
	\cdot
	\left(\int_{\Omega_s} \left((f^+)^n + \epsilon\right)^{\frac{p}{n}} \beta^n\right)^{\frac{n^2}{(n + 1) p^2}}
	\left(\int_{\Omega_s} \left((f^+)^n + \epsilon\right) \beta^n\right)^{1 - \frac{n}{(n + 1) p}} 
	\\
	&\leq
	C 	\left(\int_{\Omega_s} \left((f^+)^n + \epsilon\right)^{\frac{p}{n}} \beta^n\right)^{\frac{n^2}{(n + 1) p^2}}
		\left(\int_{\Omega_s} \left((f^+)^n + \epsilon\right) \beta^n\right)^{1 - \frac{n}{(n + 1) p}} 
		,
\end{aligned}
\end{equation}
where $C$ is independent of $\epsilon$. 
The last line in \eqref{inequality-3-31} is due to \eqref{inequality-3-20}.
Substituting \eqref{inequality-3-31} into \eqref{inequality-3-30},
\begin{equation*}
	A_s 
	\leq \frac{C}{\lambda}
	\left(\int_{\Omega_s} \left((f^+)^n + \epsilon\right)^{\frac{p}{n}} \beta^n\right)^{\frac{n}{ p^2}}
			\left(\int_{\Omega_s} \left((f^+)^n + \epsilon\right) \beta^n\right)^{1 + \frac{1}{n} - \frac{1}{p}} . 
\end{equation*}
For any $s', s > 0$,
\begin{equation*}
\begin{aligned}
	&\quad s' \int_{\Omega_{s + s'}} \left((f^+)^n + \epsilon\right) \beta^n \\
	&\leq \frac{C}{\lambda}
		\left(\int_{\Omega_0} \left((f^+)^n + \epsilon\right)^{\frac{p}{n}} \beta^n\right)^{\frac{n}{ p^2}}
				\left(\int_{\Omega_s} \left((f^+)^n + \epsilon\right) \beta^n\right)^{1 + \frac{1}{n} - \frac{1}{p}} 
				. 
\end{aligned}
\end{equation*}
By De Giorgi iteration~(Lemma~\ref{lemma-3-1}),
\begin{equation*}
\int_{\Omega_d} \left( (f^+)^n + \epsilon\right) \beta^n = 0 ,
\end{equation*}
when 
\begin{equation*}
d \geq 
2^{\frac{np + p - n}{p - n}} \frac{C}{\lambda}
		\left(\int_{\Omega_0} \left((f^+)^n + \epsilon\right)^{\frac{p}{n}} \beta^n\right)^{\frac{n}{ p^2}} \left(\int_{\Omega_0} \left( (f^+)^n + \epsilon\right) \beta^n\right)^{\frac{1}{n} - \frac{1}{p}}
		.
\end{equation*}
Letting $\epsilon \to 0+$,
\begin{equation*}
\begin{aligned}
	- u
	&\leq
	\frac{C}{\lambda} 		
	\left(\int_{\Omega_0}  (f^+)^p   \beta^n\right)^{\frac{n}{ p^2}} 
	\left(\int_{\Omega_0}  (f^+)^n  \beta^n\right)^{\frac{1}{n} - \frac{1}{p}}
	\\
	&\leq 
	\frac{C}{\lambda}
	\left(\int_{\Omega_0}  (f^+)^p   \beta^n\right)^{\frac{n}{ p^2}} 
	\left(\int_{\Omega_0}  (f^+)^p  \beta^n\right)^{\frac{n}{p}\left(\frac{1}{n} - \frac{1}{p}\right)} \left(\int_{\Omega_0} \beta^n \right)^{\left(1 - \frac{n}{p}\right)\left(\frac{1}{n} - \frac{1}{p}\right)} \\
	&\leq
	\frac{C }{\lambda}
	\left(\int_{\Omega_0} (f^+)^p \beta^n\right)^{\frac{1}{p}}
	.
\end{aligned}
\end{equation*}

\end{proof}

\begin{remark}

Replacing the auxiliary complex Monge-Amp\`ere equation~\eqref{equation-3-3} by a real Monge-Amp\`ere equation, then we can derive Alexandroff-Bakelman-Pucci estimate with the $L^\infty$-estimate by Wang~\cite{Wang2009}. 
The proof will be simpler as the De Giorgi iteration is not in need, while we need to require $f^+ \in L^n$ where $n$ is the real dimension.

\end{remark}

\medskip
\section{Harnack inquality}
\label{Harnack}

In this section, we shall prove Harnack inequalities for viscosity solutions.  
By Lemma~\ref{lemma-2-7} and a standard argument, the H\"older regularity can be derived from a Harnack inequality. 
The proofs in this section closely follow Caffarelli-Cabre~\cite{CaffarelliCabre} and Han-Lin~\cite{HanLin1}. 
For completeness and confidence (coefficients, dependence etc.), we write down some proofs in details here for later use, e.g. the calculation of modulus of continuity. 

%
%
%

%

We denote the cube centered at origin with length $r$ by $Q_r$. Then the following inclusion sequence holds true:
\begin{equation*}
	B_{\frac{1}{4}} (\bm{0})\subset B_{\frac{1}{2}} (\bm{0})\subset Q_1 \subset Q_3 \subset B_{2\sqrt{2n}} (\bm{0}).
\end{equation*}
Function $g$ is defined by
\begin{equation*}
		g (z) := - M \left( 1 - \frac{|\bm{z}|^2}{8n}\right)^\gamma \qquad \text{ in } B_{2\sqrt{2n}}(\bm{0})
\end{equation*}
where 
$ \gamma := (128n - 1) (n - 1) \frac{\Lambda}{\lambda} + 128 n$, 
and $M > 0$ is chosen  large enough that
$ g \leq - 2 \text{ in } Q_3 $.

Suppose that $f$ is a continuous bounded function and $u$ belongs to $\overline{\mathcal{S} } (\lambda,\Lambda,f)$  in  $B_{2 \sqrt{2n} }(\bm{0})$, we shall consider the following function 
\begin{equation*}
w := u + g \qquad \text{ in } B_{2\sqrt{2n}} (\bm{0}) .
\end{equation*}
Suppose that $\varphi$ is a $C^2$ function such that $w - \varphi$ has a local minimum at $\bm{z_0} \in B_{2\sqrt{2n}} (\bm{0})$,
we derive from the definition that
\begin{equation*}
\begin{aligned}
	&
	\mathcal{M}^- \left(\sqrt{-1} \partial\bar\partial \varphi (\bm{z_0}), \lambda, \Lambda\right) - \mathcal{M}^+ \left(\sqrt{-1} \partial\bar\partial g (\bm{z_0}), \lambda, \Lambda\right) \\
	&\leq
	\mathcal{M}^- \left(\sqrt{-1} \partial\bar\partial (\varphi - g) (\bm{z_0}),\lambda,\Lambda\right) \leq f (\bm{z_0}) .
\end{aligned}
\end{equation*}
We calculate the complex Hessian matrix of $g$,
\begin{equation*}
\begin{aligned}
	\partial_i \bar\partial_j g
	&= 
	\frac{M \gamma}{8 n} \left(1 - \frac{|\bm{z}|^2}{8n}\right)^{\gamma - 2} \left(\left(1 - \frac{|\bm{z}|^2}{8n}\right) \delta_{ij} - \frac{\gamma - 1}{8 n} \bar z^i z^j\right)
	,
\end{aligned}
\end{equation*}
and hence the eigenvalue set of $\sqrt{-1} \partial\bar\partial g$ can be stated as 
\begin{equation*}
\lambda_1 = \cdots = \lambda_{n - 1} = \frac{M \gamma}{8 n} \left(1 - \frac{|\bm{z}|^2}{8n}\right)^{\gamma - 1} > 0
\end{equation*}
and
\begin{equation*}
\lambda_n 
= \frac{M \gamma}{8 n} \left(1 - \frac{|\bm{z}|^2}{8n}\right)^{\gamma - 2} \left( 1  - \frac{\gamma }{8n} |\bm{z}|^2\right)
.
\end{equation*}
When $\lambda_n \leq 0$, 
\begin{equation*}
\begin{aligned}
	&\quad \mathcal{M}^+ (\sqrt{-1}\partial\bar\partial g (\bm{z}) , \lambda , \Lambda) \\
	&=
	\frac{M \gamma}{8 n} \left(1 - \frac{|\bm{z}|^2}{8n}\right)^{\gamma - 2} \left((n - 1) \Lambda \left(1 - \frac{|\bm{z}|^2}{8 n}\right) + \lambda \left(1 - \frac{\gamma}{8n} |\bm{z}|^2\right)\right) ,
\end{aligned}
\end{equation*}
and as a result
\begin{equation*}
	\mathcal{M}^+ (\sqrt{-1}\partial\bar\partial g (\bm{z}), \lambda, \Lambda) \leq 0, \qquad \forall |\bm{z}| \geq \frac{1}{4}.
\end{equation*}
It implies that if $|\bm{z_0}| \geq \frac{1}{4}$,
\begin{equation*}
	\mathcal{M}^- \left(\sqrt{-1} \partial\bar\partial \varphi (\bm{z_0}), \lambda, \Lambda\right) 
	\leq f (\bm{z_0}) ,
\end{equation*} 
and meanwhile if $|\bm{z_0}| < \frac{1}{4}$,
\begin{equation*}
	\mathcal{M}^- \left(\sqrt{-1} \partial\bar\partial \varphi (\bm{z_0}), \lambda, \Lambda\right) 
	\leq 
	f (\bm{z_0}) + \mathcal{M}^+ \left(\sqrt{-1} \partial\bar\partial g (\bm{z_0}), \lambda, \Lambda\right) 
	.
\end{equation*}
Therefore, we conclude that
\begin{equation*}
	w \in \overline{\mathcal{S}} \left( \lambda, \Lambda, f + \mathcal{M}^+ \left(\sqrt{-1} \partial\bar\partial g  , \lambda, \Lambda\right) \right) 
	\qquad \text{ in } B_{2\sqrt{2n}} (\bm{0}) .
\end{equation*}
In particular, $\mathcal{M}^+ \left(\sqrt{-1} \partial\bar\partial g  , \lambda, \Lambda\right) \in C_0 (B_{\frac{1}{4}} (\bm{0}) ) \cap Lip (B_{2\sqrt{2n}} (\bm{0}))$ and 
\begin{equation}
\label{inequality-4-15}
0 \leq \mathcal{M}^+ \left(\sqrt{-1} \partial\bar\partial g  , \lambda, \Lambda\right) \leq  \frac{M \gamma \Lambda}{8}  .
\end{equation}

\medskip

\subsection{Modulus of continuity}

\begin{lemma}
\label{lemma-4-2}
Let $f$ be a continuous and bounded function. 
Suppose that  $p > n$ and $u \in  \overline{\mathcal{S} } (\lambda,\Lambda,f)$ in  $B_{2 \sqrt{2n} }(\bm{0})$.  
Then there are constants $c_0 \in (0,1)$, $\mu \in (0,1)$ and $M > 1$ depending on $n$, $p$, $\lambda$ and $\Lambda$, 
such that 
\begin{equation*}
|\{ u < M\} \cap Q_1| > \mu ,
\end{equation*}
given that
\begin{equation*}
\inf_{B_{2\sqrt{2n}} (\bm{0})} u \geq 0, \qquad 
\inf_{Q_3} u \leq 1, \qquad
\text{ and } \; 
\int_{B_{2\sqrt{2n}} (\bm{0})} (f^+)^n \ln^p (1 + (f^+)^n) \beta^n\leq c_0 .
\end{equation*}

\end{lemma}

\begin{proof}

Since
\(
\inf_{B_{2\sqrt{2n}} (\bm{0})} w \leq -1 
\)
and
\(
w \geq 0
\)
on $\partial B_{2\sqrt{2n}} (\bm{0})$, 
\begin{equation}
\label{inequality-4-18}
\begin{aligned}
	1 
	&\leq
	C	\left( \int_{\left\{ w (\bm{z}) < 0\right\}}  \left( (f^+ + \eta)^n \right)   \ln^p \left(1 +  (f^+ + \eta)^n   \right)  \beta^n	+ 1	\right)^{\frac{1}{p}} \\
	&\qquad \cdot \left(\int_{\left\{ w (\bm{z}) < 0\right\}} (f^+ + \eta)^n   \beta^n\right)^{\frac{1}{n} - \frac{1}{p}} \\
	&\leq 
	C	\left( \int_{\left\{ w (\bm{z}) < 0\right\}}  (f^+)^n \left(  \ln^p (1 + (f^+)^n )   + 1 \right) \beta^n	+ 1	\right)^{\frac{1}{p}}  \\
	&\qquad \cdot \left(\left(\int_{\left\{ w (\bm{z}) < 0\right\}} (f^+)^n   \beta^n\right)^{\frac{1}{n}} + \left(\int_{\left\{ w (\bm{z}) < 0\right\} \cap B_{\frac{1}{4}} (\bm{0})}  \beta^n\right)^{\frac{1}{n}}\right)^{1 - \frac{n}{p}} \\
	&\leq 
	C	\left( \int_{B_{2\sqrt{2n} } (\bm{0})}  (f^+)^n  \ln^p (1 + (f^+)^n ) \beta^n	 +  \int_{B_{2\sqrt{2n} } (\bm{0})}  (f^+)^n  \beta^n	 + 1	\right)^{\frac{1}{p}}  \\
	&\qquad \cdot \left(\left(\int_{B_{2\sqrt{2n}} (\bm{0})} (f^+)^n   \beta^n\right)^{\frac{1}{n}} + \left(\int_{\left\{ w (\bm{z}) < 0\right\} \cap B_{\frac{1}{4}} (\bm{0})}  \beta^n\right)^{\frac{1}{n}}\right)^{1 - \frac{n}{p}}
	,
\end{aligned}
\end{equation}
by Theorem~\ref{theorem-3-2} and Inequality~\eqref{inequality-4-15}. 

Observing that
\begin{equation*}
\begin{aligned}
	&\quad \int_{B_{2\sqrt{2n}} (\bm{0})} (f^+)^n \ln^p (1 + (f^+)^n) \beta^n \\
	&\geq
	\ln^p (1 + c_0^{\frac{1}{p + 1}}) \int_{\{f^+ > c_0^{\frac{1}{n (p + 1)}}\}} (f^+)^n \beta^n +  \int_{\{f^+ \leq c_0^{\frac{1}{n (p + 1)}}\}} (f^+)^n \ln^p (1 + (f^+)^n) \beta^n \\
	&\geq 
	\ln^p (1 + c_0^{\frac{1}{p + 1}}) \int_{B_{2\sqrt{2n}} (\bm{0})} (f^+)^n \beta^n    - \ln^p (1 + c_0^{\frac{1}{p + 1}})  \int_{\{f^+ \leq c_0^{\frac{1}{n (p + 1)}}\}} (f^+)^n\beta^n \\
	&\geq
	\ln^p (1 + c_0^{\frac{1}{p + 1}}) \int_{B_{2\sqrt{2n}} (\bm{0})} (f^+)^n \beta^n -   (16 n)^n n! \omega_{2n} c_0^{\frac{1}{p + 1}} \ln^p (1 + c_0^{\frac{1}{p + 1}} ) 
	,
\end{aligned} 
\end{equation*}
we obtain that
\begin{equation}
\label{inequality-4-20}
\begin{aligned}
	\int_{B_{2\sqrt{2n} } (\bm{0})} (f^+)^n \beta^n 
	&\leq  \frac{c_0}{\ln^p (1 + c_0^{\frac{1}{p + 1}})} +   (16n)^{n} n! \omega_{2n} c_0^{\frac{1}{p + 1}} 
	&\leq 
	\left(2^p + (16n)^n n! \omega_{2n}\right) c_0^{\frac{1}{p + 1}} 
	,
\end{aligned} 
\end{equation}
where $\omega_{2n}$ is the volume of $B_1 ({\bm{0}})$.
Substituting \eqref{inequality-4-20} into \eqref{inequality-4-18},
\begin{equation*}
\begin{aligned}
	1 
	&\leq 
	C   \left(\left(2^p + (16n)^n n! \omega_{2n}\right)^{\frac{1}{n}} {c_0}^{\frac{1}{n(p + 1)}}  +  \left(\int_{\left\{ w (\bm{z}) < 0\right\} \cap B_{\frac{1}{4}} (\bm{0})}  \beta^n\right)^{\frac{1}{n}}\right)^{1 - \frac{n}{p}} \\
	&\leq 
	C \left(c^{\frac{1}{n(p + 1)}}_0 + \left|\left\{w(\bm{z}) < 0\right\} \cap B_{\frac{1}{4}} (\bm{0})\right|^{\frac{1}{n}}\right)^{1 - \frac{n}{p}}
	,
\end{aligned}
\end{equation*}
and then taking the $\frac{p}{p-n}$-th power,
\begin{equation*}
\begin{aligned}
	1 
	&\leq
	C \left(c^{\frac{1}{n(p + 1)}}_0 + \left|\left\{w(\bm{z}) < 0\right\} \cap B_{\frac{1}{4}} (\bm{0})\right|^{\frac{1}{n}}\right) .
\end{aligned}
\end{equation*}
Choosing $c_0 \in (0,1)$ so small that
\begin{equation*}
	C {c_0}^{\frac{1}{n(p + 1)}} < \frac{1}{2} ,
\end{equation*} 
we obtain that
\begin{equation}
\label{inequality-4-24}
	C \left|\left\{w(\bm{z}) < 0\right\} \cap B_{\frac{1}{4}} (\bm{0})\right|^{\frac{1}{n}} > \frac{1}{2}. 
\end{equation}
In $B_{2\sqrt{2n}} (\bm{0})$, 
\begin{equation*}
	w (\bm{z}) 
	\geq u (\bm{z}) - M ,
\end{equation*}
and thus we derive from \eqref{inequality-4-24} that
\begin{equation*}
\left|\{ u (\bm{z})< M\} \cap Q_1\right|  \geq \left|\{u (\bm{z}) < M\} \cap B_{\frac{1}{4}} (\bm{0})\right| > \mu  
\end{equation*}
for some $\mu > 0$.

\end{proof}

We shall show by induction that
\begin{equation}
\label{inequality-4-27}
	\left| \{u(\bm{z}) \geq M^k\} \cap Q_1 \right| \leq (1 - \mu)^k
\end{equation}
for $k = 1,2,\cdots$, where $M$ and $\mu$ are as in Lemma~\ref{lemma-4-2}.                                         
%
%
%
%
Lemma~\ref{lemma-4-2} ensures that  Inequality~\eqref{inequality-4-27} holds true for $k = 1$.
Suppose that \eqref{inequality-4-27} holds true for $k - 1$, then we define
\begin{equation*}
	A := \{u (\bm{z}) \geq M^k\} \cap Q_1, \qquad B := \{u(\bm{z}) \geq M^{k - 1}\} \cap Q_1 ,
\end{equation*}
and plan to prove $|A| \leq (1 - \mu) |B|$.
It suffices to prove the following lemma to apply Calderon-Zygmund decomposition argument 
since 
$|A| \leq 1 - \mu$.
\begin{lemma}
\label{lemma-4-2-1}
If $  Q_r (\bm{z_0}) \subset Q_1$ 
satisfies
\begin{equation}
\label{inequality-4-36}
	|A \cap  Q_r (\bm{z_0})| > (1 - \mu) | Q_r (\bm{z_0})| ,
\end{equation}
then 
\begin{equation*}
	Q_{3r} (\bm{z_0}) \cap Q_1 \subset B = \{u(\bm{z}) \geq M^{k - 1}\} \cap Q_1 .
\end{equation*}
\end{lemma}
\begin{proof}

Suppose that 
\begin{equation*}
	Q_{3r} (\bm{z_0}) \cap Q_1 \not\subset B.
\end{equation*}
We may pick $\bm{z_1} \in Q_{3r} (\bm{z_0})$ with $u (\bm{z_1}) < M^{k - 1}$. 
Define
\begin{equation*}
\bm{z} := \bm{z_0} + r\bm{w} 
\end{equation*}
and
\begin{equation*}
	\tilde u (\bm{w}) := \frac{1}{M^{k - 1}} u(\bm{z}) .
\end{equation*}
It is easy to see that $r \leq 1 - 2 \Vert\bm{z_0}\Vert_{max} := 1 - 2 \max_i |z^i_0|$. 
Then
\begin{equation*}
	\bm{z} 
	= \bm{z_0} + r \bm{w} 
	\in \bm{z_0} + r B_{2\sqrt{2n}} (\bm{0}) 
	\subset \bm{z_0} + \left(1 - 2 \Vert\bm{z_0}\Vert_{max}\right) B_{2\sqrt{2n}} (\bm{0}) 
	\subset  B_{2\sqrt{n}} (\bm{0})  
	,
\end{equation*}
and 
\begin{equation*}
	\bm{z} 
	= \bm{z_0} + r \bm{w} 
	\in \bm{z_0} + r Q_3 
	\subset \bm{z_0} + \left(1 - 2 \Vert\bm{z_0}\Vert_{max}\right) Q_3
	\subset  Q_3
	.
\end{equation*}
So $\tilde u (\bm{w}) \geq 0$ in $B_{2\sqrt{2n}} (\bm{0})$ and $\inf_{Q_3} \tilde u (\bm{w}) \leq 1$.

In the viscosity sense,
\begin{equation*}
\begin{aligned}
	\mathcal{M}^- (\sqrt{-1} \partial_{\bm{w}}\bar\partial_{\bm{w}} \tilde u (\bm{w}), \lambda, \Lambda)
	&=  \mathcal{M}^- \left(\sqrt{-1} \partial_{\bm{w}}\bar\partial_{\bm{w}} \left(\frac{1}{M^{k - 1}} u \left(\bm{z_0} + r \bm{w}\right)\right), \lambda, \Lambda\right) \\
	&\leq 
	\frac{r^2}{M^{k - 1}} f (\bm{z_0} + r\bm{w}),
\end{aligned}
\end{equation*}
and 
\begin{equation*}
\begin{aligned}
	&\quad
	\int_{B_{2\sqrt{2n}} (\bm{0})} \left(\frac{r^2}{M^{k - 1}} f(\bm{z_0} + r \bm{w})^+\right)^n \ln^p \left(1 + \left(\frac{r^2}{M^{k - 1}} f(\bm{z_0} + r \bm{w})^+\right)^n\right) \beta^n (\bm{w}) \\
	&\leq
	\int_{B_{2\sqrt{2n}} (\bm{0})} \frac{1}{M^{n(k - 1)}} \left( f(\bm{z})^+\right)^n \ln^p \left(1 + \left(\frac{r^2}{M^{k - 1}} f(\bm{z})^+\right)^n\right) \beta^n (\bm{z}) \\
	&\leq 
	\frac{1}{M^{n(k - 1)}} \int_{B_{2\sqrt{2n}} (\bm{0})} \left( f(\bm{z})^+\right)^n \ln^p \left(1 + \left(f(\bm{z})^+\right)^n\right) \beta^n (\bm{z}) \\
	&\leq c_0 
	.
\end{aligned}
\end{equation*}
By Lemma~\ref{lemma-4-2}, 
\begin{equation*}
	|\{\tilde u (\bm{w}) < M\} \cap Q_1| > \mu,
\end{equation*}
while
\begin{equation*}
\begin{aligned}
|\{\tilde u  (\bm{w})< M\} \cap Q_1| 
&= 
\frac{1}{r^{2n}} \left|\left\{  u (\bm{z}) < M^k\right\} \cap Q_r (\bm{z_0})\right|
.
\end{aligned}
\end{equation*}
Then
\begin{equation*}
\left|\left\{  u (\bm{z}) < M^k\right\} \cap Q_r (\bm{z_0})\right| > \mu r^{2n} ,
\end{equation*}
which contradicts \eqref{inequality-4-36}.

\end{proof}
%
%
%
%
%
%
%
%
%
%
%
%
%
By Lemma~\ref{lemma-4-2-1} and Calderon-Zygmund decomposition,
\begin{equation*}
|A| \leq (1 - \delta) |B| .
\end{equation*}
By induction,
\begin{equation*}
	\left| \{u(\bm{z}) \geq M^k\} \cap Q_1 \right| \leq (1 - \mu)^k .
\end{equation*}
Taking $C = \frac{1}{1 - \mu}$ and $c_1 = - \frac{\ln (1 - \mu)}{\ln M}$, we obtain the following lemma.

\begin{lemma}
\label{lemma-4-3}
Let $f$ be a continuous and bounded function. 
Suppose that $p > n$ and $u \in  \overline{\mathcal{S} } (\lambda,\Lambda,f)$ in  $B_{2 \sqrt{2n} }(\bm{0})$. 
Then there are positive constants $c_0 $, $c_1$ and $C$, depending  on $n$, $p$, $\lambda$ and $\Lambda$, such that for any $t > 0$, 
\begin{equation*}
	\left|\{ u \geq t\} \cap Q_1\right| \leq C t^{- c_1},
\end{equation*}
given that
\begin{equation*}
	\inf_{B_{2\sqrt{2n}} (\bm{0})} u \geq 0 , \qquad \inf_{Q_3} u \leq 1, \qquad \text{ and }\; \int_{B_{2\sqrt{2n}} (\bm{0})} (f^+)^n \ln^p (1 + (f^+)^n) \beta^n \leq c_0 .
\end{equation*}

\end{lemma}

Now we use the above lemma to obtain a Harnack inequality. 
\begin{lemma}
\label{lemma-4-6}
Let $f$ be a continuous and bounded function. 
Suppose that $p > n$ and $u \in  {\mathcal{S} } (\lambda,\Lambda,f)$ in  $Q_{4\sqrt{2n}}$ . 
Then there are positive constants $c_0$  and $C$, depending on $n$, $p$, $\lambda$ and $\Lambda$, such that $\sup_{Q_{\frac{1}{4}}} u \leq C$, 
in the event that
\begin{equation*}
	\inf_{Q_{4\sqrt{2n}}} u \geq 0,  \qquad \inf_{Q_{\frac{1}{4}}} u \leq 1, \qquad \text{and} \qquad \int_{Q_{4\sqrt{2n}}} |f|^n \ln^p (1 + |f|^n) \beta^n \leq c_0 ,
\end{equation*}

\end{lemma}

\begin{proof}

We define
\begin{equation*}
\theta := 2^{- \frac{c_1 + 1}{c_1}} C^{- \frac{1}{c_1}} + 1
\end{equation*}
and 
\begin{equation*}
	M_0 
	:= 
	2 \left( 2^{\frac{8n + 1 }{2n}} n C^{\frac{1}{2n}} \sum^\infty_{k = 1} \theta^{- (k - 1) \frac{c_1}{2n}} \right)^{\frac{2n}{c_1}} 
	+ \frac{1}{\theta - 1} 
\end{equation*}
where $c_1$ and $C$ are from Lemma~\ref{lemma-4-3}.
We shall show that 
there exists a sequence $\{\bm{z_k}\} \subset B_{\frac{1}{2}} (\bm{0})$ with
\begin{equation*}
	u (\bm{z_k}) \geq \theta^k P \qquad \text{for } k = 0 , 1, 2, \cdots,
\end{equation*}
if $u (\bm{z_0}) = P > M_0$ for some $\bm{z_0} \in B_{\frac{1}{4}} (\bm{0})$
which contradicts the boundedness of $u$. 

Suppose that $u (\bm{z_0}) = P > M_0$ for some $\bm{z_0} \in B_{\frac{1}{4}} (\bm{0})$.
Since
\begin{equation*}
	\inf_{Q_3} u \leq \inf_{Q_{\frac{1}{4}}} u \leq 1,
\end{equation*}
Lemma~\ref{lemma-4-3} implies that
\begin{equation}
\label{inequality-4-50}
	\left| \left\{ u > \frac{P}{2}\right\} \cap Q_1\right| 
	\leq C \left(\frac{P}{2}\right)^{- c_1} 
	< C \left(\frac{M_0}{2}\right)^{- c_1} 
	.
\end{equation}
Choosing 
\begin{equation*}
r = (2 C)^{\frac{1}{2n}} \left(\frac{M_0}{2}\right)^{- \frac{c_1}{2n}} < \frac{1}{16 n},
\end{equation*} 
we have
$
	Q_r (\bm{z_0}) \subset Q_1
$
and
\begin{equation}
\label{inequality-4-58}
	\frac{1}{|Q_r (\bm{z_0})|} \left| \left\{ u > \frac{P}{2}\right\} \cap Q_r (\bm{z_0})\right| 
	<
	\frac{1}{2}.
\end{equation}

Next, we shall prove by contradiction that  $u \geq \theta P$ at some point in $Q_{4\sqrt{2n} r} (\bm{z_0})$.
Suppose that $u < \theta P$ in $Q_{4\sqrt{2n} r} (\bm{z_0})$. 
For $ \bm{w} \in Q_{4\sqrt{2n}} $ and $\bm{z} \in Q_{4\sqrt{2n} r} (\bm{z_0})$, we define
\begin{equation*}
	\bm{z} := \bm{z_0} + r\bm{w} ,
\end{equation*}
and
\begin{equation*}
	\tilde u (\bm{w}) := \frac{\theta P - u (\bm{z})}{(\theta - 1) P}.
\end{equation*}
It is easy to see that
\begin{equation*}
	\tilde u (\bm{w}) \geq \frac{\theta P - \theta P}{(\theta - 1) P} = 0, \qquad \text{for } \bm{w} \in Q_{4\sqrt{2n}}
\end{equation*}
and
\begin{equation*}
	\tilde u (\bm{0}) = \frac{\theta P - u (\bm{z_0})}{(\theta - 1) P} = \frac{\theta P - P}{(\theta - 1) P} = 1.
\end{equation*}
In the viscosity sense,
\begin{equation*}
\begin{aligned}
	\mathcal{M}^- (\sqrt{-1} \partial_{\bm{w}}\bar\partial_{\bm{w}} \tilde u (\bm{w}), \lambda, \Lambda)
	&\leq 
	- \frac{r^2}{(\theta -1) P} f (\bm{z_0} + r\bm{w}),
\end{aligned}
\end{equation*}
and similarly
\begin{equation*}
	\mathcal{M}^+ (\sqrt{-1} \partial_{\bm{w}}\bar\partial_{\bm{w}} \tilde u (\bm{w}), \lambda, \Lambda) 
	\geq
	- \frac{r^2}{(\theta - 1) P} f(\bm{z_0} + r\bm{w}) 
	. 
\end{equation*}
Moreover,
\begin{equation*}
\begin{aligned}
	&\quad
	\int_{B_{2\sqrt{2n}} (\bm{0})} \left(\frac{r^2}{(\theta - 1) P} | f(\bm{z_0} + r \bm{w})|\right)^n \ln^p \left(1 + \left(\frac{r^2}{(\theta - 1) P} | f(\bm{z_0} + r \bm{w})|\right)^n\right) \beta^n (\bm{w}) \\
	&\leq 
	\frac{1}{(\theta - 1)^n P^n} \int_{B_{2\sqrt{2n}} (\bm{0})}  |f(\bm{z})|^n \ln^p \left(1 +  \left(\frac{r^2}{(\theta - 1) P} | f(\bm{z})| \right)^n\right) \beta^n (\bm{z}) \\
	&\leq 
 	\int_{B_{2\sqrt{2n}} (\bm{0})}  |f(\bm{z})|^n \ln^p \left(1 +   | f(\bm{z})|^n\right) \beta^n (\bm{z})  \\
	&
	\leq c_0 
	,
\end{aligned}
\end{equation*}
since $(\theta - 1) P > (\theta - 1) M_0 > 1$.
%
Applying Lemma~\ref{lemma-4-3} to $\tilde u$,
\begin{equation*}
	\frac{1}{Q_r (\bm{z_0})} \left| \left\{ u \leq \frac{P}{2}\right\} \cap Q_r (\bm{z_0})\right|
	= 
	\left|\left\{ \tilde u \geq \frac{2\theta - 1}{2\theta - 2}\right\} \cap Q_1\right| 
	\leq 
	C\left( \frac{2\theta - 1}{2\theta  -  2}\right)^{-c_1}
	<
	\frac{1}{2} ,
\end{equation*}
which contradicts \eqref{inequality-4-58}. 
%
%
%
%
%
%
%
%
%
%
%
Therefore, there exists a point $\bm{z_1} \in Q_{4\sqrt{2n} r} (\bm{z_0}) \subset B_{4nr} (\bm{z_0})$ with $u(\bm{z_1}) \geq \theta P$.

By iteration, we can obtain a sequence $\{\bm{z_k}\}$ such that
\begin{equation*}
u (\bm{z_k}) \geq \theta^k P \qquad \text{for some } \bm{z_k} \in B_{4n r_k} (\bm{z_{k - 1}})
\end{equation*}
where
\begin{equation*}
r_k  = (2 C)^{\frac{1}{2n}} \left(\frac{M_0 }{2}\right)^{- \frac{c_1}{2n}} \theta^{- \frac{(k - 1) c_1}{2n}}
.
\end{equation*}
It can be seen that $\{\bm{z_k}\} \subset B_{\frac{1}{2}} (\bm{0})$,
by noting that
\begin{equation*}
\begin{aligned}
	\sum^{\infty}_{k = 1} 4n r_k 
	&= 
	4n \sum^\infty_{k = 1} (2 C)^{\frac{1}{2n}} \left(\frac{M_0 }{2}\right)^{- \frac{c_1}{2n}} \theta^{- \frac{(k - 1) c_1}{2n}} \\
	&<
	4n \sum^\infty_{k = 1} (2 C)^{\frac{1}{2n}}  \left( 2^{\frac{8n + 1 }{2n}} n C^{\frac{1}{2n}} \sum^\infty_{k = 1} \theta^{- (k - 1) \frac{c_1}{2n}} \right)^{-1}  \theta^{- \frac{(k - 1) c_1}{2n}} 
	= \frac{1}{4}
	.
\end{aligned}
\end{equation*}

\end{proof}

As in \cite{Luxemburg1955}\cite{Adams1975}, we define
\begin{equation*}
	A \left(c_0, f, \Omega\right)
	:= 
	\inf \left\{K > 0 \Bigg| \int_{\Omega} \frac{|f|^n}{K^n} \ln^p \left(1 + \frac{|f|^n}{K^n}\right) \beta^n \leq c_0\right\}
	.
\end{equation*}
If $\Omega_1 \subset \Omega_2$,
\begin{equation*}
\int_{\Omega_1} \frac{|f|^n}{K^n} \ln^p \left(1 + \frac{|f|^n}{K^n}\right) \beta^n 
\leq
\int_{\Omega_2} \frac{|f|^n}{K^n} \ln^p \left(1 + \frac{|f|^n}{K^n}\right) \beta^n 
,
\end{equation*}
and hence
\begin{equation*}
	A(c_0,f,\Omega_1) \leq A(c_0,f,\Omega_2) .
\end{equation*}
Consider for $\delta > 0$,
\begin{equation*}
u_\delta 
:= 
\frac{u}{\inf_{Q_{\frac{1}{4}}} u + \delta +	A \left(c_0, f, Q_{4\sqrt{2n}}\right)} .
\end{equation*}
Applying Lemma~\ref{lemma-4-6} to $u_\delta$,  we obtain
\begin{equation*}
	\sup_{Q_{\frac{1}{4}}} u \leq C \left(\inf_{Q_{\frac{1}{4}}} u + \delta +	A \left(c_0, f, Q_{4\sqrt{2n}}\right)\right) .
\end{equation*}
Letting $\delta \to 0 +$,
\begin{equation}
\label{inequality-4-70}
	\sup_{Q_{\frac{1}{4}}} u \leq C \left(\inf_{Q_{\frac{1}{4}}} u   +	A \left(c_0, f, Q_{4\sqrt{2n}}\right)\right) .
\end{equation}
Then a Harnack inequality follows from a standard covering argument and \eqref{inequality-4-70}. 
\begin{theorem}
\label{theorem-4-7}
Let $f$ be a continuous and bounded function. 
Suppose that $u \in  {\mathcal{S} } (\lambda,\Lambda,f)$ in  $B_1(\bm{0})$ is nonnegative. 
Then it holds true that
\begin{equation*}
	\sup_{B_{\frac{1}{2}} (\bm{0})} u \leq C \left\{ \inf_{B_{\frac{1}{2}} (\bm{0})} u + A \left(c_0, f, B_1(\bm{0})\right) \right\} 
\end{equation*}
where $C$ is a positive constant depending on $p$, $n$, $\lambda$ and $\Lambda$.

\end{theorem}

For any $0 < r < 1$, we define
\begin{equation*}
\bm{z}: = r\bm{w}
\end{equation*}
for $w \in B_1 (\bm{0})$ and $\bm{z} \in B_r (\bm{0})$, and hence
\begin{equation*}
\tilde u (\bm{w}) := u (\bm{z}) = u (r \bm{w}) .
\end{equation*}
Then $\tilde u \in \mathcal{S} (\lambda,\Lambda, r^2 f(r\bm{w}))$ in $B_1 (\bm{0})$.
By Theorem~\ref{theorem-4-7}, 
\begin{equation}
\label{inequality-4-74}
	\sup_{B_{\frac{r}{2}} (\bm{0})} u \leq C \left\{ \inf_{B_{\frac{r}{2}} (\bm{0})}  u + A \left(c_0, r^2 f(r\bm{w}), B_1(\bm{0})\right) \right\} .
\end{equation}
We calculate, 
\begin{equation*}
\begin{aligned}
&\quad
A \left(c_0, r^2 f(r\bm{w}), B_1(\bm{0})\right) \\
&=
\inf \left\{K > 0 \Bigg| \int_{B_1 (\bm{0})} \frac{r^{2n}|f (r \bm{w})|^n}{K^n} \ln^p \left(1 + \frac{r^{2n} |f(r\bm{w})|^n}{K^n}\right) \beta^n(\bm{w}) \leq c_0\right\} \\
&=
\inf \left\{K > 0 \Bigg| \frac{1}{K^n}\int_{B_r (\bm{0})} |f (\bm{z})|^n  \ln^p \left(1 + \frac{r^{2n} |f(\bm{z})|^n}{K^n}\right) \beta^n(\bm{z}) \leq c_0\right\} 
.
\end{aligned}
\end{equation*}
Since $c_0 > 0$ and for any fixed $K > 0$
\begin{equation*}
\begin{aligned}
&\quad \frac{1}{K^n}\int_{B_r (\bm{0})} |f (\bm{z})|^n  \ln^p \left(1 + \frac{r^{2n} |f(\bm{z})|^n}{K^n}\right) \beta^n(\bm{z}) \\
&\leq 
\frac{1}{K^n}\int_{B_r (\bm{0})} |f (\bm{z})|^n  \ln^p \left(1 + \frac{ |f(\bm{z})|^n}{K^n}\right) \beta^n(\bm{z}) 
\to 0 +
,
\end{aligned}
\end{equation*}
we obtain that
\begin{equation*}
A \left(c_0, r^2 f(r\bm{w}), B_1(\bm{0})\right) \to 0 + .
\end{equation*}
For $0 < r_1 < r_2$ and fixed $K > 0$,
\begin{equation*}
\begin{aligned}
&\quad \frac{1}{K^n}\int_{B_{r_1} (\bm{0})} |f (\bm{z})|^n  \ln^p \left(1 + \frac{{r_1}^{2n} |f(\bm{z})|^n}{K^n}\right) \beta^n(\bm{z}) \\
&\leq 
\frac{1}{K^n}\int_{B_{r_2} (\bm{0})} |f (\bm{z})|^n  \ln^p \left(1 + \frac{{r_2}^{2n} |f(\bm{z})|^n}{K^n}\right) \beta^n(\bm{z}) 
,
\end{aligned}
\end{equation*}
we know that $A \left(c_0, r^2 f(r\bm{w}), B_1 (\bm{0})\right)$ is nondecreasing.

Applying \eqref{inequality-4-74} to $\sup_{B_r (\bm{0})} u - u$ and $u - \inf_{B_r (\bm{0})} $ respectively, 
\begin{equation}
\label{inequality-4-79}
	 \sup_{B_r (\bm{0})} u - \inf_{B_{\frac{r}{2}} (\bm{0})} u 
	 \leq  C \left\{\inf_{B_r (\bm{0})}  u - \sup_{B_{\frac{r}{2}} (\bm{0})}  u + A \left(c_0, r^2 f(r\bm{w}), B_1(\bm{0})\right) \right\} ,
\end{equation}
and
\begin{equation}
\label{inequality-4-80}
	 \sup_{B_{\frac{r}{2}} (\bm{0})} u - \inf_{B_r (\bm{0})} u 
	 \leq  C \left\{\inf_{B_{\frac{r}{2}} (\bm{0})}  u - \sup_{B_r (\bm{0})}  u + A \left(c_0, r^2 f(r\bm{w}), B_1(\bm{0})\right) \right\} .
\end{equation}
Putting \eqref{inequality-4-79} and \eqref{inequality-4-80} together, 
\begin{equation}
\label{inequality-4-12-1}
	(C + 1) osc_{B_{\frac{r}{2}} (\bm{0})} u \leq (C - 1) osc_{B_r (\bm{0})} u + 2 C A \left(c_0, r^2 f(r\bm{w}), B_1(\bm{0})\right)  .
\end{equation}
By rescaling, we can also obtain a Liouville theorem if $f \equiv 0$ in $\mathbb{C}^n$.
%
%
%
%
%
%
%
%
%
%
%
%
%
%
%
%
%
%
%
For any $\mu \in (0,1)$ and $r \leq R$,
\begin{equation}
\label{inequality-4-82}
	osc_{B_r (\bm{0})} u \leq C \left( r^\alpha osc_{B_1 (\bm{0})} u + A \left(c_0, r^{2\mu} f(r^\mu\bm{w}), B_1(\bm{0})\right)  \right) ,
\end{equation}
by applying Lemma~\ref{lemma-2-7} to \eqref{inequality-4-12-1}. 

Now we need to know the modulus of continuity from $ A \left(c_0, r^{2\mu} f(r^\mu\bm{w}), B_1(\bm{0})\right) $.
For $p > q > n$ and $0 < r < 1$,
\begin{equation}
\label{inequality-4-83}
\begin{aligned}
	&\quad
	\int_{B_r (\bm{0})} |f|^n \ln^p (1 + |f|^n) \beta^n \\
	&\geq
	\ln^{p - q} \left(1 + \frac{1}{r^n}\right)\int_{B_r(\bm{0}) \cap \{|f| > \frac{1}{r}\}} |f|^n \ln^q (1 + |f|^n) \beta^n \\
	&\geq 
	\ln^{p - q} \left(1 + \frac{1}{r^n}\right)\int_{B_r(\bm{0})} |f|^n \ln^q (1 + |f|^n) \beta^n \\
	&\qquad - \ln^{p - q} \left(1 + \frac{1}{r^n} \right)\int_{B_r(\bm{0}) \cap\{|f| \leq \frac{1}{r}\}} \frac{1}{r^n} \ln^q \left(1 + \frac{1}{r^n} \right) \beta^n \\
	&\geq 
	\ln^{p - q} \left(1 + \frac{1}{r^n}\right) \int_{B_r(\bm{0})} |f|^n \ln^q (1 + |f|^n) \beta^n  -  2^n n! \omega_{2n} r^{n} \ln^p \left(1 + \frac{1}{r^n}\right)  
	.
\end{aligned}
\end{equation}
Rearranging \eqref{inequality-4-83}
\begin{equation*}
\begin{aligned}
	&\quad
	\int_{B_r (\bm{0})} |f|^n \ln^p (1 + |f|^n) \beta^n + 2^n n! \omega_{2n} r^{n}  \ln^p \left(1 + \frac{1}{r^n}\right)  \\
	&\geq 
	\ln^{p - q} \left(1 + \frac{1}{r^n}\right)\int_{B_r(\bm{0})} |f|^n \ln^q (1 + |f|^n) \beta^n \\
		&\geq
		n^{p - q} (- \ln r)^{p - q} \int_{B_r(\bm{0})} |f|^n \ln^q (1 + |f|^n) \beta^n 
		.
\end{aligned}
\end{equation*}
and hence
\begin{equation*}
\begin{aligned}
	\int_{B_r(\bm{0})} |f|^n \ln^q (1 + |f|^n) \beta^n 
	&\leq   	\frac{1}{n^{p - q}	\left(-\ln  r\right)^{p - q}} 	\int_{B_1 (\bm{0})} |f|^n \ln^p (1 + |f|^n) \beta^n  \\
	&\qquad  + 2^{n + q - 1} n! \omega_{2n} r^{n}  \left(\ln^q 2 + n^q (- \ln r)^q\right) 
	.
\end{aligned}
\end{equation*}
Then,
\begin{equation*}
	\frac{C(n,p,q)}{(- \ln r)^{p - q}} 	\int_{B_1 (\bm{0})} |f|^n \ln^p (1 + |f|^n) \beta^n + C (n,q) \sqrt{r} \geq 	\int_{B_r(\bm{0})} |f|^n \ln^q (1 + |f|^n) \beta^n .
\end{equation*}
In conclusion,
\begin{equation*}
\begin{aligned}
	\int_{B_r(\bm{0})} |f|^n \ln^q (1 + |f|^n) \beta^n 
	&\leq
	\frac{C (n,p,q)}{(- \ln r)^{p - q}} \left(\int_{B_1 (\bm{0})} |f|^n \ln^p (1 + |f|^n) \beta^n + 1\right) 
	\\
	&\leq \frac{C_2 }{(-\ln r)^{p - q}}
.
\end{aligned}
\end{equation*}
When  $r > 0$ small enough,
\begin{equation}
	\frac{c_0 (- \ln r)^{p - q} r^{2n}}{C_2} < 1 ,
\end{equation}
and consequently
\begin{equation*}
\begin{aligned}
	c_0
	&\geq 
	\frac{c_0 (- \ln r)^{p - q} }{C_2}
	\int_{B_r (\bm{0})} |f|^n \ln^q \left(1 +  |f|^n \right) \beta^n
	\\
	&\geq \int_{B_r (\bm{0})} \frac{c_0 (- \ln r)^{p - q} |f|^n}{C_2} \ln^q \left(1 + \frac{c_0 (- \ln r)^{p - q} r^{2n} |f|^n}{C_2 }\right) \beta^n
	.
\end{aligned}
\end{equation*}
Therefore
\begin{equation}
\label{inequality-4-90}
A (c_0, r^2 f(r\bm{w}),B_1(\bm{0})) \leq \frac{C^{\frac{1}{n}}_1}{c^{\frac{1}{n}}_0 (- \ln r)^{\frac{p - q}{n}}} .
\end{equation}
In fact, we can also reach the result via the generalized H\"older inequality for Orlicz space.
Substituting \eqref{inequality-4-90} into \eqref{inequality-4-82}, we obtain the modulus of continuity.
\begin{theorem}
Let $f$ be a continuous and bounded function. 
Suppose that $u \in  {\mathcal{S} } (\lambda,\Lambda,f)$ in  $B_1(\bm{0})$ is nonnegative. 
For any $  0 < k < \frac{p - n}{n}$, we have
\begin{equation*}
	osc_{B_r (\bm{0})} u \leq \frac{C}{(- \ln r)^k}  \left(osc_{B_1 (\bm{0})} u + \int_{B_1 (\bm{0})} |f|^n \ln^p (1 + |f|^n) \beta^n  + 1 \right) ,
\end{equation*}
where $C$ is a positive constant depending  on $p$, $n$, $k$, $\lambda$ and $\Lambda$. 
\end{theorem}

\medskip

\subsection{H\"older regularity}

When $f$ is in $L^p$, we can lift modulus of continuity up to H\"older regularity. The proofs are almost the same, and the results seem very similar to those in Caffarelli-Cabre~\cite{CaffarelliCabre} and Han-Lin~\cite{HanLin1}. 
So we only state the results to clearly elaborate the exponents and coefficients, but omit the details of arguments. We have the following Harnack inequalities.

\begin{theorem}
\label{theorem-4-10}
Let $f$ be a continuous and bounded function. 
Suppose that $u \in  {\mathcal{S} } (\lambda,\Lambda,f)$ in  $B_1(\bm{0})$ is nonnegative. 
Then it holds true that
\begin{equation*}
	\sup_{B_{\frac{1}{2}} (\bm{0})} u \leq C \left\{ \inf_{B_{\frac{1}{2}} (\bm{0})} u + \Vert f \Vert_{L^p (B_1 (\bm{0}))}\right\} 
\end{equation*}
where $C$ is a positive constant depending on $p$, $n$, $\lambda$ and $\Lambda$.

\end{theorem}

\begin{theorem}
\label{theorem-4-13}
Let $f$ be a continuous and bounded function. 
Suppose that $u \in  \underline{\mathcal{S} } (\lambda,\Lambda,f)$ in $B_1 (\bm{0})$. 
Then it holds true that
\begin{equation}
\label{inequality-4-112-1}
	\sup_{B_{\frac{1}{2}} (\bm{0})} u \leq C \left\{ \Vert u^+ \Vert_{L^q (B_{\frac{3}{4}} (\bm{0}))} + \Vert f \Vert_{L^p (B_1 (\bm{0}))}\right\} 
\end{equation}
where $C$ is a positive constant depending on $p$, $q$, $n$, $\lambda$ and $\Lambda$.

\end{theorem}

Let $u \in \mathcal{S} (\lambda, \Lambda,f)$ be nonnegative in $B_1 (\bm{0})$. For any $0 < r < 1$, we define 
$\bm{z}: = r\bm{w}$  
and hence
$\tilde u (\bm{w}) := u (\bm{z}) = u (r \bm{w})$,  
for $\bm{w} \in B_1 (\bm{0})$ and $\bm{z} \in B_r (\bm{0})$. 
Then we can see that $\tilde u \in \mathcal{S} (\lambda,\Lambda, r^2 f(r\bm{w}))$ in $B_1 (\bm{0})$.
By Theorem~\ref{theorem-4-10}, 
\begin{equation*}
	\sup_{B_{\frac{r}{2}} (\bm{0})} u 
	\leq 
	C \left\{ \inf_{B_{\frac{r}{2}} (\bm{0})}  u + \left\Vert r^2 f (r\bm{w}) \right\Vert_{L^p \left(B_1(\bm{0})\right)} \right\} 
	=
	C \left\{ \inf_{B_{\frac{r}{2}} (\bm{0})}  u + r^{\frac{2 (p - n)}{p}} \Vert f\Vert_{L^p (B_r (\bm{0}))} \right\} 
	.
\end{equation*}
%
%
%
Applying Lemma~\ref{lemma-2-7} as previous, we can prove H\"older regularity.
\begin{theorem}
Let $f$ be a continuous and bounded function. 
Suppose that $p > n$ and $u \in  {\mathcal{S} } (\lambda,\Lambda,f)$ in  $B_1(\bm{0})$ is nonnegative. 
Then there are constants $\alpha \in (0,1)$ and $C > 0$ depending on $p$, $n$, $\lambda$ and $\Lambda$, it holds true that
\begin{equation}
\label{inequality-4-117}
	osc_{B_r (\bm{0})} u \leq C r^\alpha \left(osc_{B_1 (\bm{0})} u + \Vert f\Vert_{L^p (B_1 (\bm{0}))} \right) .
\end{equation}
\end{theorem}

\begin{proof}[Proof of Theorem~\ref{theorem-1-3}]
Substituting \eqref{inequality-4-112-1} with $q = 2$ into \eqref{inequality-4-117} after appropriate rescaling, it suffices to adapt a standard rescaling and covering argument to prove Theorem~\ref{theorem-1-3}.

\end{proof}

\medskip

\noindent
{\bf Acknowledgements}\quad
The author wish to thank Chengjian Yao for helpful discussions and suggestions. 
The author is supported by a start-up grant from ShanghaiTech University (2018F0303-000-04).


\end{document}